\documentclass{article}

      \usepackage{amsmath}
      \usepackage{amsfonts}
      \usepackage{amssymb}
      \usepackage{amsthm}

      \def\om{\omega}
      \def\aut#1{\operatorname{Aut(#1)}}
      \def\be{\operatorname{\beta}}
      \def\dim{\operatorname{dim}}
      \def\al{\alpha}

      \def\cp{\mathcal{P}}
      \def\cl#1{\overline{#1}}
      \def\clx#1#2{\overline{#2}^{#1}}
      \def\int#1{\operatorname{int} (#1)}
      \def\intx#1#2{\operatorname{int}_{#1} (#2)}
      \def\id{\operatorname{id}}
      \def\sp{\operatorname{sp}}

      \def\es{\varnothing}
      \def\Tau{{\mathcal T}}
      
      \def\sset#1{\{#1\}}
      \def\set#1{\bbset#1\eeset}
      \def\bbset#1:#2\eeset{\{#1\,:\,#2\}}
      
      \def\bbsett#1:#2\eesett{\{#1\,:\,\text{#2}\}}
      
      \def\go{{\mathfrak{O}}} 
      \def\gs{{\mathfrak{S}}}
      
      \def\cm{{\mathcal M}}
      
      \def\cpcc/{CPCC}
      \def\ed/{ED}
      \def\edd/{ED${}_d$}

      \def\lkr{\leq_{KR}}

      \newtheorem{assertion}{Statement}
      \newtheorem{proposition}{Proposition}
      \newtheorem{theorem}{Theorem}
      \newtheorem{lemma}{Lemma}
      \newtheorem{cor}{Corollary}
      
      \theoremstyle{definition}
      \newtheorem{definition}{Definition}
      \newtheorem{question}{Question}

      \title{Homogeneous Subspaces of Products of Extremally Disconnected Spaces}
      \author{Evgenii Reznichenko\thanks{Supported by RFBR, project no. 17-51-18051 Bolg\_a}}
\date{}

\begin{document}
      
      \maketitle

      \section{Introduction}

      In \cite{cr1966} it was proved that the product of pseudocompact groups 
is pseudocompact. In \cite{cm1985} the  compactness-type properties of products  
of homogeneous spaces were studied. In particular,  homogeneous 
pseudocompact spaces $X$ and $Y$ whose product $X\times Y$ is 
not pseudocompact was constructed \cite[Theorem 1.1]{cm1985}. 
      Under Martin's axiom  
homogeneous countably compact spaces $X$ and $Y$ whose product 
$X\times Y$ is not pseudocompact were also constructed \cite[Theorem 4.1]{cm1985} 
      and the problem of the 
existence of such spaces without additional set-theoretic assumptions was 
posed \cite[Question 5.1(a)]{cm1985}) . 
      In this paper we construct such an example (Theorem~\ref{nphscp}). However, the 
problem of the existence of a homogeneous pseudocompact space $X$ for 
which the product $X\times X$ is not pseudocompact \cite[Question 5.3]{cm1985} still remains 
open. 
      
      In \cite{cm1987} products of homogeneous extremally 
disconnected spaces were studied. Under Martin's axiom 
homogeneous extremally disconnected countably compact spaces $X$ and $Y$ with non-countably compact 
product $X\times Y$ were constructed \cite[Theorem 4.2]{cm1987}. 
      In \cite{ls1997} this example was improved; namely, the 
product $X\times Y$  was made nonpseudocompact \cite[Theorem 3]{ls1997}. 
      In \cite{k1994} 
      homogeneous extremally disconnected countably compact spaces $X$ and 
$Y$ for which the product $X\times Y$ is not countably compact were constructed  
  without additional set-theoretic assumptions~\cite[Theorem 4.1]{k1994}. 
      
      The homogeneous extremally disconnected countably compact spaces 
constructed in \cite{cm1987,k1994,ls1997} are {\it maximally homogeneous} 
(see Definition~\ref{defmaxhom} below). In this paper 
we prove that the countable power of a maximally homogeneous extremally 
disconnected countably compact space is countably compact (Corollary~\ref{mh6}).
      
      In \cite{arh67} Arhangel'skii proved that any compact subset of an extremally 
disconnected topological group is finite. At the same time, Frol\'{\i}c proved that homogeneous 
extremally disconnected compact spaces are finite~\cite{fro1967}. In this paper we obtain a 
simultaneous generalization of these two theorems; namely, we prove that 
all compact subsets of homogeneous extremally disconnected spaces are finite (Corollary~\ref{tc1}), 
which answers Questions~4.5.2 and~4.5.3 in~\cite{at2009}.
      This statement follows also from Theorem~2(c) of~\cite{Douwen-1979}.
      It can be strengthened: all compact subsets of homogeneous
      subspaces of the third power of an extremally disconnected space 
are finite (Theorem~\ref{t4}). The author is unaware of whether this 
is true for the fourth power without additional set-theoretic assumptions (Question~\ref{q2}). 
However, under CH, it is true. Moreover,
      under CH, all compact subsets of homogeneous
      subspaces of any finite power of an extremally disconnected space 
are finite (Theorem~\ref{t3}) and all compact subsets  
of homogeneous
      subspaces of the countable power of an extremally disconnected space 
are metrizable (Theorem~\ref{t2}).
      It is unknown whether the last statement can be proved na\"{\i}vely 
(Question~\ref{q1}).
      
      We also strengthen Frol\'{\i}k's theorem mentioned above by proving that all compact homogeneous 
subspaces of finite powers of an extremally disconnected space are finite (Theorem~\ref{t5}).

      \section{Conventions, Definitions, and Notation}
      
      All spaces considered in this paper are assumed to be completely regular and Hausdorff.
      The set of positive integers is denoted by $\om$, $\be\om$ is the set of all ultrafilters on 
$\om$, and $\om^*=\be\om\setminus\om$ is the set of free ultrafilters on~$\om$.
      Given a space $X$ and a set $M\subset X$, by $\clx XM$ and $\intx XM$ we denote, 
respectively, the closure and the interior of $M$ in $X$. When it is clear from the context what 
ambient space $X$ is meant, we write simply $\cl M$ and $\int M$
      instead of $\clx XM$ and $\intx XM$.
      By  $\id_X$ we denote the identity self-map of~$X$.
      
      Recall that a space $X$ is said to be {\it extremally disconnected} 
(\emph{\ed/}) if
      $\cl U\cap\cl V=\es$ for any disjoint open sets 
$U,V\subset X$.
      A completely regular Hausdorff space is called 
      an \emph{F-space} if every cozero set is $C^*$-embedded in this space.
      Elementary properties of F-spaces and \ed/ spaces can be found in the book
       \cite{gli-jer-1960} by Gillman and Jerison.
      Any \ed/ space is an F-space.
      
      In  a space $X$ is called a {$\be\om$ 
space} if, given any countable discrete set $M\subset 
X$ with compact closure $\cl M$, its closure 
      is homeomorphic  to $\be\om$~\cite{van-douwen-1981}.
      Any $F$-space is a $\be\om$ space, and 
      any subspace of a $\be\om$ space is a $\be\om$ space.
      
      Let us remind the definition of the  Keisler--Rudin order $\lkr$ on $\om^*$.
      Let $p,q\in \om^*$. Then $p\lkr q$ if and only if there exists
      a map $f\colon\om\to\om$ such that $\beta f(q)=p$.
      Here $\beta f\colon \be \om\to \be\om$ is the continuous extension of the map~$f$.
      
      Let $X$ be space. We say that a sequence $\zeta=(x_n)_{n\in\om}$ on $X$ is 
      {\it exact} if $x_n\ne x_m$ for $n\ne m$ and {\it discrete} if 
the set  $\set{x_n: n\in\om}$ is discrete.
      
      We use the following notation: 
      \begin{itemize}
      \item
      $\gs(X)=X^\om$ is the set of all sequences on $X$;
      \item
      $\gs_k(X)$ is the set of sequences on $X$ whose ranges have compact closure in~$X$; 
      \item
      $\gs_d(X)$ is the set of discrete exact sequences on $X$;
      \item
      $\gs_{dk}(X)=\gs_{d}(X)\cap \gs_{k}(X)$ is the set of discrete exact 
sequences on $X$ whose ranges have compact closure.
      \end{itemize}

      Let $\Tau$ denote the topology of $X$, and let $\Tau_*=\Tau\setminus\sset{\es}$.
      Then 
      \begin{itemize}
      \item
       $\go(X)=\Tau_*^\om$ is the set of sequences of nonempty open 
subsets of~$X$;
      \item
      $\go_d(X)$ is the set of disjoint sequences of nonempty open 
subsets of~$X$.
      \end{itemize}
      
      Given a  sequence $\xi=(M_n)_{n\in\om}$ of subsets of $X$, a point $x\in 
X$, and a set $M\subset X$, we put
      
      \begin{itemize}
      \item
      \begin{align*} 
      u_x(\xi)=\{p\in\om^*:\,& \set{n\in\om: U\cap M_n\ne\es}\in p 
      \\
      &\mbox{ for each neighborhood $U$ of $x$}\};
      \end{align*}
      \item
      $u_M(\xi)=\bigcup_{x\in M}u_x(\xi)$;
      \item
      $u(\xi)=u_X(\xi)$.
      \end{itemize}
      
      For a sequence $\zeta=(x_n)_{n\in\om}$ of points on~$X$, 
      we set 
      
      \begin{itemize}
      \item
      \begin{align*} 
      u_x(\zeta)=\{p\in\om^*:\,& \set{n\in\om: x_n\in U}\in p 
      \\
      &\mbox{ for each neighborhood $U$ of $x$}\};
      \end{align*}
      \item
      $u_M(\zeta)=\bigcup_{x\in M}u_x(\zeta)$;
      \item
      $u(\zeta)=u_X(\zeta)$.
      \end{itemize}

      Note that, for $\zeta'=(\sset{x_n})_{n\in\om}$, we have $u_x(\zeta)=u_x(\zeta')$,
      $u_M(\zeta)=u_M(\zeta')$, and
      $u(\zeta)=u(\zeta')$.
      
      We also set 
      \begin{itemize}
      \item
      $\sp(x,X)=\bigcup_{\zeta\in\gs_d(X)} u_x(\zeta)$;
      \item
      $\sp(X)=\bigcup_{x\in X} \sp(x,X)$;
      \item
      $\sp_k(X)=\bigcup_{\zeta\in\gs_{dk}(X)} u_x(\zeta)$;
      \item
      $\sp_k(X)=\bigcup_{x\in X} \sp_k(x,X)$.
      \end{itemize}
      
      Note that $p\in u_x(\zeta)$ if and only if $x$ is
      the {\it $p$-limit} of the sequence $\zeta$.
      We denote the $p$-limit of $\zeta$ by $\lim_p \zeta$.
      
      \section{Sequences on Spaces}
      
      \begin{definition}
      We say that a sequence $\zeta=(x_n)_{n\in\om}$ is {\it almost exact} 
if there exists an $N\in\om$ such that $x_n\ne x_m$ for $n\ne m$ and $n,m>N$.
      We say that $\zeta$ is {\it almost stationary} if 
      there exists an $N\in\om$ such that $x_n= x_m$ for $n,m>N$.
      \end{definition}
      
      \begin{assertion}\label{cp1}
      Suppose that $X=\prod_{n\in\om}X_n$ is a countable product of spaces,
      $M\subset X$ is an infinite subset of $X$, and $\pi_n\colon X\to X_n$ is the projection for 
      each $n\in \om$. Then there exists an exact sequence $(x_k)_{k\in\om}\subset M$ 
such that $(\pi_n(x_k)_{k\in\om}$ is an almost stationary or almost exact 
discrete sequence in $X_n$ for each~$n$.
      \end{assertion}
      
\begin{proof}
It is easy to construct, by induction, a sequence of infinite 
sets $(M_n\subset M)_{n\in\om}$ such that, for each $n$, $M_{n+1}\subset M_n$ and
      either $|\pi_n(M_n)|=1$ or $\pi_n\restriction_{M}$ is an injective 
map and the set $\pi_n(M_n)$ is discrete in $X_n$.
      It remains to take an  exact sequence $(x_k)_{k\in\om}$  such that $x_k\in 
M_k$ for all~$k$.
      \end{proof}
      
      \begin{assertion}\label{a1n1}
      Let $X$ be a space, and let $(U_n)_{n\in\om}\in\go(X)$. 
      Then there exists an infinite set $M\subset\om$ and 
      a family of nonempty open sets $(V_n)_{n\in M}$ such that
      $V_{n}\subset U_n$ for $n\in M$
      and  one of  the following  conditions holds:
      \begin{itemize}
      \item[\rm(1)] the family $\set{V_n:n\in M}$ is disjoint;
      \item[\rm(2)] all $V_n$ coincide and consist of one point.
      \end{itemize}
      \end{assertion}
      
\begin{proof}
Consider two cases.
      
      {\it The first case}.\enspace 
      For any finite set  $K$, the complement $U_n\setminus K$ is nonempty for 
      infinitely many~$n$. 
      Then it is easy to inductively construct an infinite  set 
      $M'\subset\om$ and a  sequence $(x_n)_{n\in M'}$ such that $x_n\in U_n$ and $x_m\neq 
x_n$  for any different $n,m\in M'$. Choose an infinite set $M\subset M'$ so that 
      $\set{x_n:n\in M}$ is discrete. Take neighborhoods $W_n$ of $x_n$ for which the family 
      $\set{W_n:n\in M}$ is disjoint. The family of $V_n=U_n\cap W_n$ satisfies condition~(1).

      {\it The second case.}\enspace 
      There exists a finite set $K\subset X$ which contains almost all $U_n$. 
In this case, we can find its subset $U\subset K$ and an infinite set $M\subset\om$ such 
that $U_n=U$ for $n\in M$. Choose $u\in U$. The family of $V_n=\sset{u}$, $n\in M$, satisfies 
condition~(2).
      \end{proof}

      \begin{assertion}\label{cp2}
      Let $X=\prod_{n\in\om}X_n$ be a countable product of spaces, 
and let $(W_k)_{k\in\om}\in \go(X)$.
      Then there exists an increasing sequence $(n_k)_{k\in\om}\subset 
\om$ and sequences 
      $(U_{n,k})_{k\in\om}\in \go(X_n)$, $n\in\om$, such that 
      $\prod_{n\in\om} U_{n,k} \subset W_{n_k}$ and, for each $n\in\om$, 
the family 
$\zeta_n=\set{U_{n,k}: k>n}$ satisfies one of  the following  conditions:
      \begin{itemize}
      \item[\rm(1)] $\zeta_n$ is disjoint;
      \item[\rm(2)] $\zeta_n$ consists of coinciding one-point sets. 
      \end{itemize}
      \end{assertion}
      
\begin{proof}
      There exist $(V_{n,k})_{k\in\om}\in \go(X_n)$, $n\in\om$, such that 
$\prod_{n\in\om} V_{n,k} \subset W_{k}$.
      
      Using Statement~\ref{a1n1}, we can easily construct 
      a decreasing sequence $(M_n)_{n\in\om}$ infinite subsets 
of $\om$ and
      a family of sequences $(Q_{n,k})_{k\in\om}\in\go(X_n)$ 
      so that $Q_{n,k}\subset V_{n,k}$ for all $n,k$ 
      and  one of  the following  conditions holds:
      \begin{itemize}
      \item[\rm(1)] $\set{Q_{n,k}:k\in M_n}$ is disjoint;
      \item[\rm(2)] $\set{Q_{n,k}:k\in M_n}$ consists of coinciding one-point 
sets.
      \end{itemize}
     It remains to take an increasing sequence $(n_k)_{k\in\om}\subset 
\om$ with $n_k\in M_k$ for $k\in \om$ and 
      put $U_{n,k}=Q_{n,n_k}$ for $k,n\in \om$.
      \end{proof}

      \begin{assertion}
      For any space $X$ and any sequences $(U_n)_{n\in\om},(V_n)_{n\in\om}\in\go_d(X)$, 
      one of  the following  conditions holds:
      \begin{itemize}
      \item[\rm(1)] $U_n=V_n$ and $|U_n|=1$ for almost all $n$;
      \item[\rm(2)]
      there exists an infinite set $M\subset\om$ and 
      families  $(U'_n)_{n\in M}$ and  $(V'_n)_{n\in M}$ of nonempty open sets
such that $\bigcup_{n\in M} U'_n\cap \bigcup_{n\in M} V'_n =\es$
      and $U'_n\subset U_n$, $V'_n\subset V_n$ for all $n\in M$.
       \end{itemize}
      \end{assertion}
      
\begin{proof}
      Suppose that (1) does not hold. Then there exists an infinite set 
$M'\subset \om$ and 
      families $(U''_n)_{n\in M}$ and $(V''_n)_{n\in 
M}$ of nonempty open sets such that $U''_n\cap V''_n=\es$ 
and $U''_n\subset U_n$, $V''_n\subset 
V_n$ for all $n\in M'$. We set 
      $A_n=\set{m\in M': U''_n\cap V''_m\neq \es,\ m>n}$ and
      $B_n=\set{m\in M': U''_m\cap V''_n\neq \es,\ m>n}$ for $n\in M'$.
      Consider three cases.
      
      {\it The first case.}\enspace 
The set $A_n$ is infinite for some $n\in M'$.   
In this case, we take $M=A_n$ and put $U'_m=U''_m$ and $V'_m=V''_m\cap U_n$ for $m\in M$.
      
      {\it The second case}.\enspace 
The set $B_n$ is infinite for some $n\in M'$.   
      In this case, we take $M=B_n$ and put $U'_m=U''_m\cap V_n$ and $V'_m=V''_m$ for $m\in M$.
      
      {\it The third case.}\enspace 
The sets $A_n$ and $B_n$ are finite for all $n\in M'$.
      In this case, there exists a sequence $(m_k)_{k\in\om}\subset M'$ such that
      $m_{k+1}>m_k$, $m_{k+1}>\max A_{m_k}$, and $m_{k+1}>\max B_{m_k}$ for all 
$k$.
      We take $M=\set{m_k:k\in \om}$ and put $U'_m=U''_m$ and
       $V'_m=V''_m\cap U_n$ for $m\in M$.
      \end{proof}
      
      \section{Extremally Disconnected Spaces\\ and Ultrafilters}
      
      \begin{assertion}[\cite{gli-jer-1960,bal-dow-1991,com-neg-1991}]
\label{ed1}
      The following assertions are equivalent for any completely regular Hausdorff space 
$X$:
      \begin{itemize} 
      \item[\rm(1)] $X$ is \ed/;
      \item[\rm(2)] every dense subset of $X$ is \ed/;
      \item[\rm(3)] every open subset of $X$ is \ed/;
      \item[\rm(4)] every dense subset of $X$ is $C^*$-embedded;
      \item[\rm(5)] every open subset of $X$ is $C^*$-embedded.
      \end{itemize} 
      \end{assertion}
      
     The following assertion is a direct consequence of definitions.
      
      \begin{assertion}\label{ed3}
      A space $X$ is a  $\be\om$ space if and only if 
$|u_x(\zeta)|\leq 1$
      for any $x\in X$ and $\zeta\in\gs_{dk}(X)$.
      \end{assertion}

      \begin{proposition}\label{ed4}
      Let $X$ be a space.
      \begin{itemize}
      \item[\rm(1)] If $X$ contains an infinite compact set, then $\sp_k(X)=\om^*$.
      \item[\rm(2)] If $X$ is homogeneous, then $\sp(X)=\sp(x,X)$ and 
$\sp_k(X)=\sp_k(x,X)$ for any $x\in X$.
      \end{itemize}
      \end{proposition}
      
\begin{proof} 
(1)\enspace       Since $X$ contains an infinite compact set, it follows that $\gs_{dk}(X)\neq\es$.
      If $\zeta\in \gs_{dk}(X)$, then $u_X(\zeta)=\om^*$.
      
(2)\enspace  Since $X$ is homogeneous, it follows that $\sp(x,X)=\sp(y,X)$ and \tolerance2000 
$\sp_k(x,X)=\sp_k(y,X)$ for $x,y\in X$.
      \end{proof}
      
      \begin{proposition}[Frol\'{\i}k, \cite{fro1967}]
\label{ed5}
      Let $X$ be an \ed/ space, and let $x\in X$. 
      Then the set $\sp(x,X)$ is totally ordered with respect to the 
Keisler--Rudin order.
      \end{proposition}
      
      \begin{proposition}\label{ed6}
      Let $X$ be a $\be\om$ space, and let $x\in X$. 
      Then the set $\sp_k(x,X)$ is totally ordered with respect to the Keisler--Rudin order.
      \end{proposition}
      
\begin{proof}
      Let $p,q\in\sp_k(x,X)$.
      Then 
      we have $\sset p=u_x(\zeta)$ and $\sset q=u_x(\zeta)$ for some $\zeta,\xi\in\gs_{dk}(X)$.
      Let $M$ be the set of isolated points in $\zeta\cup\xi$, and let $K=\cl M$.
      Then $\zeta,\xi\subset K$ and $p,q\in \sp(x,K)$. 
      Since $K$ is homeomorphic to $\be\om$ and extremally disconnected, it follows by 
Proposition~\ref{ed5} that $p$ and $q$ are $\lkr$-comparable.
      \end{proof}

      \begin{proposition}[{Kunen, \cite[Lemma 4]{kunen-1990}}]
\label{ed7}
      Let $p, q \in \om^*$ be $\lkr$-incomparable weak $P$-points, 
      and let $X$ be any compact F-space. Suppose that 
      $\zeta=(x_n)_{n\in\om}\in\gs_{d}(X)$, 
      $\xi=(y_n)_{n\in\om}\in\gs(X)$, 
      and $x = \lim_p \zeta = \lim_q \xi$.
      Then $\set{n : y_n = x} \in q$.
      \end{proposition}
      
      Proposition~\ref{ed7} implies the following assertion. 
      
      \begin{proposition}\label{ed8}
      Let $p, q \in \om^*$ be $\lkr$-incomparable weak $P$-points, 
      and let  $X$ be any F-space. Suppose that 
      $\zeta=(x_n)_{n\in\om}\in\gs_{dk}(X)$, 
      $\xi=(y_n)_{n\in\om}\in\gs_k(X)$, and 
      $x = \lim_p \zeta = \lim_q \xi$.
      Then $\set{n : y_n = x} \in q$.
      \end{proposition}
      
      \begin{proposition}[Simon, \cite{Simon-1985}; see also
      \cite{kunen-1978,Shelah-Rudin-1979}]
\label{ed9} 
      There exists a set $C\subset \om^*$, $|C|=2^{2^\om}$, consisting of pairwise
      $\lkr$-incomparable weak $P$-ultrafilters.
      \end{proposition}
      
      \begin{proposition}[\cite{Rudin-1956}]
\label{ed10}
      \textup{(CH)}
      There exists a set $C\subset \om^*$, $|C|=2^{2^\om}$, consisting of pairwise
      $\lkr$-incomparable selective ultrafilters.
      \end{proposition}
      
      \begin{assertion}\label{ed11} 
      Suppose that $X$ is a $\be\om$ space, $A$ is a set, 
      $K$ is a compact subspace of $X^A$, and 
      one of  the following  conditions holds:
      \begin{itemize}
      \item[\rm(1)] $A$ is finite and $K$ is infinite;
      \item[\rm(2)] $A$ is infinite and $w(K)>|A|$.
      \end{itemize}
      Then $\be\om$ is embedded  in $K$.
      \end{assertion}
      
\begin{proof}
      For each $\al\in A$, let $\pi_\al: X^A\to X$ denote the projection onto the $\al$th 
factor.
      There exists an $\al\in A$ for which $\pi_\al(K)$ is infinite.
      Let $M$ be a countable discrete subspace of $K$ for which
      $\pi_\al\restriction_M$ is injective and $\pi_\al(M)$ is 
      a discrete subspace of $X$. Then $\cl M$ is homeomorphic to $\be\om$.
      \end{proof}

      \begin{assertion}\label{ed12} 
      Suppose that $X$ is an $\be\om$ space,
      $p\in\om^*$ is a selective ultrafilter, 
$\zeta=(z_n)_{n\in\om}\in\gs(X^\om)$,
and   $z=\lim_p\zeta\in X^\om$. Suppose also that 
      $\sset z\cup \set{z_n:n\in M}$ is nonmetrizable for any $M\in p$.
      Let $\pi_k:X^\om\to X$ denote the projection of $X^\om$ onto the $k$th factor.
      Then there exists an $m\in\om$ and an $N\in p$ such that
      the sequence $(\pi_m(z_n))_{n\in N}$
      is discrete and exact, i.e.,
      the set $\set{\pi_m(z_n):n\in N}$ is discrete and $\pi_m(x_j)\ne 
\pi_m(x_i)$
      for any different $i,j\in N$.
      \end{assertion}
      
\begin{proof}
      First, we take $m\in\om$ such that $\pi_m(M)$ is infinite for 
each $M\in p$. Such an $m$ exists. Indeed, otherwise, we can 
choose $M_m\in p$ so that $|\pi_m(M_m)|=1$ for each $m\in \om$.  
      Since the ultrafilter $p$ is selective, there exists an $M\in p$ such that 
$M_m\setminus M$ is finite for each $m\in \om$. The set 
      $\sset z\cup \set{z_n:n\in M}$ is metrizable, which contradicts the assumption. 
      
      Let $(U_n)_{n\in\om}$ be a sequence of neighborhoods 
of the point $\pi_m(z)$ for which $\pi_m(\zeta)\cap\bigcap_{n\in\om}U_n$ contains at 
most one point and $U_{i+1}\subset\cl{U_{i+1}}\subset U_{i}$ for 
$i\in\om$. 
Let $N_n=\set{i\in\om: \pi_m(x)i)\in U_n}$. Then $N_n\in p$. Since the ultrafilter $p$ is selective, 
we can find $N\in p$ such that $|N\cap N_i\setminus N_{i+1}|\le 1$ for 
$i\in\om$. The set $\set{\pi_m(z_n):n\in N}$ is discrete, and $\pi_m(x_j)\ne \pi_m(x_i)$ for 
different $i,j\in N$.
      \end{proof}

      \section{Countably Compact and Pseudocompact\\ Product Spaces}
      
       The following assertion is easy to verify.

      \begin{assertion}\label{ss1}
      Suppose that $\set{X_\al:\al\in A}$ is a family spaces,
      $\zeta_\al=(R_{\al,n})_{n\in\om}$ is a sequence of nonempty 
subsets of $X_\al$ for each $\al$, and  
      $R_n=\prod_{\al\in\om}R_{\al,n}$.
      Then  the sequence $(R_n))_{n\in\om}$ is locally finite (or, in other words, 
has no accumulation points) in $\prod_{\al\in A}X_\al$ if and only if 
      $\bigcap_{\al\in A} u_{X_\al}(\zeta_\al)=\es$.
      \end{assertion}

      \begin{proposition}\label{p1s}
      Let $X$ be a space, and let $\tau$ be a cardinal.
      \begin{itemize}
      \item[\rm(a)] The following conditions are equivalent:
      \begin{itemize}
      \item[\rm(1)] $X$ is countably compact;
      \item[\rm(2)]  $u(\zeta)\neq \es$ for all $\zeta\in  \gs(X)$;
      \item[\rm(3)]  $u(\zeta)\neq \es$ for all $\zeta\in \gs_d(X)$.
      \end{itemize}
      
      \item[\rm(b)] The following conditions are equivalent:
      \begin{itemize}
      \item[\rm(1)] $X^\om$ is countably compact;
      \item[\rm(2)]  $\bigcap_{\zeta\in\gamma}u(\zeta)\neq \es$ for all $\gamma 
\subset \gs(X)$ with $|\gamma|\le \om$;
      \item[\rm(3)]  $\bigcap_{\zeta\in\gamma}u(\zeta)\neq \es$ for all $\gamma 
\subset \gs_d(X)$ with $|\gamma|\le \om$.
      \end{itemize}
      
      \item[\rm(c)] The following conditions are equivalent:
      \begin{itemize}
      \item[\rm(1)] $X^\tau$ is countably compact;
      \item[\rm(2)]  $\bigcap_{\zeta\in\gamma}u(\zeta)\neq \es$ for all $\gamma 
\subset \gs(X)$ with $|\gamma|\le \tau$.
      \end{itemize}
      \end{itemize}
      \end{proposition}
      
\begin{proof}
      Statement~\ref{ss1} 
      implies the equivalence of conditions (1) and (2) in assertions (a), (b), and (c).
       The implications (2)$\implies$(3) in (a) and (b) are obvious. 
      The implication (3)$\implies$(2) in (a) follows from the fact that each countable set 
contains a discrete countable subset.
      
      Let us prove that (3)$\implies$(2) in (b).
      Let $\pi_n$ denote the projection of $X^\om$ onto the $n$th factor.
      Take an infinite set $M\subset X$.
      Statement~\ref{cp1} implies the existence of an exact 
sequence $(x_k)_{k\in\om}\subset M$ such that
      $(\pi_n(x_k)_{k\in\om}$ is an almost stationary or almost exact 
discrete sequence in $X$.
      In view of (3), we have $P=\bigcap_{\zeta\in\gamma}u(\zeta)\neq \es$, where 
$\gamma=\set{(\pi_n(x_k)_{k\in\om}: n\in\om}$. Take $p\in P$.
      Let $y_n$ be the $p$-limit of the sequence $\pi_n(x_k)_{k\in\om}$, 
and let $y=(y_n)_{n\in\om}$.
      Then the point $y$ is the $p$-limit of $(x_k)_{k\in\om}$. Therefore, 
$y$ is an accumulation point for the set~$M$.
      \end{proof}

      \begin{proposition}\label{p1p}
      Let $X$ be a space, and let $\tau$ be an infinite cardinal.
      \begin{itemize}
      \item[\rm(a)] The following conditions are equivalent:
      \begin{itemize}
      \item[\rm(1)] $X$ is pseudocompact;
      \item[\rm(2)]  $u(\zeta)\neq \es$ for all $\zeta\in  \go(X)$;
      \item[\rm(3)]  $u(\zeta)\neq \es$ for all $\zeta\in \go_d(X)$.
      \end{itemize}
      
      \item[\rm(b)] The following conditions are equivalent:
      \begin{itemize}
      \item[\rm(1)] $X^\om$ is pseudocompact;
      \item[\rm(2)]  $\bigcap_{\zeta\in\gamma}u(\zeta)\neq \es$ for all $\gamma 
\subset \go(X)$ with $|\gamma|\le \om$;
      \item[\rm(3)]  $\bigcap_{\zeta\in\gamma}u(\zeta)\neq \es$ for all $\gamma 
\subset \go_d(X)$ with $|\gamma|\le \om$;
      \item[\rm(4)] $X^\tau$ is pseudocompact.
      \end{itemize}
      \end{itemize}
      \end{proposition}
      
\begin{proof}
Assertion (a) follows from Statement~\ref{a1n1}.
      Let us prove (b).
      
      (1)$\implies$(2)
      Let $\gamma=\set{(U_{n,k})_{k\in\om}: n\in\om}\subset \go(X)$.
      We set 
      \[
      V_{n,k}=
      \begin{cases}
      X,          &  k<n\\
      U_{n,k},    &  k\geq n
      \end{cases}
      \]
      Note that $u((U_{n,k})_{k\in\om})=u((V_{n,k})_{k\in\om})$.
      The set $V_k=\prod_{n\in\om}V_{n,k}$ is open in $X^\om$.
      It follows from the pseudocompactness of $X^\om$ that the family $(V_k)_{k\in\om}$ is not 
locally finite. Statement~\ref{ss1} implies 
      $\bigcap_{\zeta\in\gamma}u(\zeta) = 
\bigcap_{n\in\om}u((V_{n,k})_{k\in\om})  \neq \es$.
      
      (2)$\implies$(1).
      Let $(W_k)_{k\in\om}\in\go(X^\om)$.
      There exist sequences $(V_{n,k})_{k\in\om}\in\go(X)$ such that 
$\prod_{n\in\om}V_{n,k}\subset V_k$. 
      By virtue of Statement~\ref{ss1} implies that $(V_k)_{k\in\om}$ 
      is not locally finite.
      
      The implication (2)$\implies$(3) is obvious.
      
      (3)$\implies$(1).
      Let $(W_k)_{k\in\om}\in\go(X^\om)$.
      By virtue of Statement~\ref{cp2} 
      there exists an increasing sequence $(n_k)_{k\in\om}\subset 
\om$ and sequences $(U_{n,k})_{k\in\om}\in \go(X)$, $n\in\om$, such that 
      $\prod_{n\in\om} U_{n,k} \subset W_{n_k}$ and, for each $n\in\om$, 
the family 
$\zeta_n=\set{U_{n,k}: k>n}$
      satisfies one of  the following conditions:
      \begin{itemize}
      \item[\rm(1)] $\zeta_n$ is disjoint;
      \item[\rm(2)] $\zeta_n$ consists of coinciding one-point sets.
      \end{itemize}
      According to (3), we have $\bigcap_{n\in\om}u((U_{n,k})_{k\in\om})\neq \es$.
      By virtue of Statement~\ref{ss1} the sequence $(W_k)_{k\in\om}$ is not locally 
finite.
      
      The equivalence (3)$\iff$(4) was proved in~\cite{Gli}.
      \end{proof}

      \begin{definition}
      Let $p\in\om^*$ be an ultrafilter.
      We say that a space $X$ is {\it discretely $p$-compact} 
      if  any discrete exact sequence 
      $(x_n)_{n\in\om}\subset X$
      has a $p$-limit.
      \end{definition}
      
      \begin{definition}
      Let $p\in\om^*$ be an ultrafilter.
      We say that a space $X$ is {\it sequentially $p$-compact} if 
      any infinite set $M\subset X$ has an infinite subset $L\subset M$
      such that any exact sequence $(x_n)_{n\in\om}\subset L$
      has a $p$-limit.
      \end{definition}
      
      Any infinite set contains an choose infinite 
discrete subset. This implies the following assertion. 
      
      \begin{proposition}\label{dsps}
      Let $p\in\om^*$ be an ultrafilter.
      Then any discrete $p$-compact space
      is sequentially $p$-compact.
      \end{proposition}

      \begin{proposition}\label{cpscp}
      Let $X$ be a sequentially $p$-compact space.
      Then $X^\om$ is sequentially $p$-compact.
      \end{proposition}
      
\begin{proof}
Take an infinite set $M\subset X^\om$. By virtue of Statement~\ref{cp1} 
      there exists an exact sequence $(x_k)_{k\in\om}\subset M$ such that
      $(\pi_n(x_k)_{k\in\om}$ is an almost stationary or almost exact 
discrete sequence in $X$ for each $n$ (here $\pi_n:X\to X$ denotes the projection 
onto the $n$th factor).
      There exists a decreasing sequence $(S_n)_{n\in\om}$ of infinite 
subsets of $\om$ such that, for each $n\in\om$,  one of  the following  conditions holds:
      \begin{itemize}
      \item[\rm(1)] $\pi_n(x_i)=\pi_n(x_j)$ for $i,j\in S_n$;
      \item[\rm(2)] $\pi_n(x_i)\neq \pi_n(x_j)$ for any different $i,j\in S_n$
      and each exact sequence $(y_l)_{l\in\om}\subset 
\set{\pi_n(x_k): k\in S_n}$ has a $p$-limit.
      \end{itemize}
      Take an exact sequence $(s_n)_{n\in\om}$ such that $s_n\in 
S_n$ for all $n$. Let $L=\set{x_{s_k}: k\in\om}$.
      Then any exact sequence $(y_n)_{n\in\om}\subset L$ 
has a $p$-limit.
      \end{proof}
      
      \begin{proposition}\label{cedcs}
      Let $X$ be a compact $\be\om$ space, and let $M\subset X$,
      $|M| < 2^{2^\om}$. Then  $Y=X\setminus M$ 
      is a discrete $p$-compact space for some 
$p\in\om^*$.
      \end{proposition}
      
\begin{proof}
      Suppose that, on the contrary, 
      there exist $\zeta_p\in\gs_d(Y)$ such that $\lim_p \zeta_p\notin 
Y$ for all $p\in\om^*$.
      By virtue of Statement~\ref{ed9}, there exists
      a set $C\subset \om^*$, $|C|=2^{2^\om}$, which consists of pairwise
      $\lkr$-incomparable ultrafilters.
      Since $|C|>|M|$,  it follows that $x=\lim_p \zeta_p=\lim_q \zeta_q\in M$ 
      for some different $p,q\in C$. We have $p,q\in\sp_k(x,X)$, 
      which contradicts Proposition~\ref{ed6}.
      \end{proof}

      \section{Homogeneous Product Spaces}
      
      Given a cardinal $\tau$, a set $A$, a family of spaces $(X_\al)_{\al\in A}$, and a 
point  $(x_\al)_{\al\in A}\in \prod_{\al\in A}X_\al$, we set 
      We denote
      \begin{align*}
      \sigma_\tau \set{(X_\al,x_\al):\al\in A}&=
      \set{
      	(y_\al)_{\al\in A}\in \prod_{\al\in A}X_\al: 
      	|\set{\al\in A: x_\al\neq y_\al}|<\tau
      },
      \\
      \Sigma_\tau \set{(X_\al,x_\al):\al\in A}&=\sigma_{\tau^+} 
\set{(X_\al,x_\al):\al\in A}. 
      \end{align*}
      
      We also set 
      \[
      H_\tau(X)= \sigma_\tau \set{(X,x): (x,\al)\in X\times \tau }
      \]
for any space $X$. 
      The space $H_{\om_1}(X)$ is 
      a $\Sigma$-product in $X^{X\times\om_1}$, 
      and $H_{\om}(X)$ is 
      a $\sigma$-product in $X^{X\times\om}$.
      
      \begin{proposition}[\cite{rezn}]
\label{erezn1}
      Let $X$ be a space, and let $\tau$ be an infinite cardinal.
      Then the space $H_\tau(X)$ is homogeneous and homeomorphic to $X\times 
H_\tau(X)$.
      \end{proposition}
      
      \begin{proposition}[\cite{rezn}]
\label{erezn2}
      Let $X$ be a space, and let $Y=H_{\om_1}(X)$. 
      If $X$ is $p$-compact for some $p\in\om^*$, then so is $Y$.
      \end{proposition}
      
      \begin{proposition}\label{erezn3}
      Let $X$ be a space, and let $Y=H_{\om_1}(X)$.
      If $X$ is sequentially $p$-compact for some  $p\in\om^*$, then so is~$Y$.
      \end{proposition}

      \begin{proof}
      Let $M\subset Y$, $|M|=\om$. There exists a closed set 
$F\subset Y$ such that $M\subset F$ and $F$ is homeomorphic to $X^\om$.
      Proposition~\ref{cpscp} implies  the sequential $p$-compactness of $F$.
      Hence there exists an infinite set $L\subset M$
      such that any exact sequence $(x_n)_{n\in\om}\subset L$
      has a $p$-limit in $F\subset Y$.
      \end{proof}

      \begin{proposition}\label{npscp}
      For any $p\in\om^*$, there exist extremally disconnected spaces  
$X$ and $Y$ such that $X$ is $p$-compact,
       $Y$ is sequentially $q$-compact for some $q\in\om^*$, 
       and $X\times Y$ is not pseudocompact.
      \end{proposition}
      
\begin{proof}
      Let $X$ be a minimal $p$-compact subspace of $\be\om$ 
      containing $\om$. Then $|X|\leq 2^\om$. We set 
$Y=\om\cup(\be\om\setminus X)$. Since $X\cap Y=\om$, it follows that $X\times Y$ is not 
pseudocompact, and since $|\beta\om\setminus Y|\le|X|\le 2^\om$, it follows by 
Proposition~\ref{cedcs}
      that $Y$ is sequentially $q$-compact for some 
$q\in\om^*$.
      \end{proof}

      \begin{theorem}\label{nphscp}
      For any $p\in\om^*$, there exist homogeneous spaces $X$ and $Y$ such that $X$ is $p$-compact,
      $Y$ is sequentially $q$-compact for some $q\in\om^*$, and 
$X\times Y$ is not pseudocompact.
      Moreover, the space $X^\tau$ is countably compact for any $\tau$
      and $Y^\om$ is countably compact.
      \end{theorem}
      
\begin{proof}
      By virtue of Proposition~\ref{npscp}, 
      there exist extremally disconnected spaces $X'$ and $Y'$ such that $X'$ 
is $p$-compact,
      ${Y'}$ is sequentially $q$-compact for some $q\in\om^*$, and 
$X'\times Y'$ is not pseudocompact.
      We set $X=H_{\omega_1}(X')$ and $Y=H_{\omega_1}(Y')$.
      According to Proposition~\ref{erezn1}, the spaces $X$ and $Y$ are homogeneous 
and $X'\times Y'$ is a continuous image of $X\times Y$. 
Therefore, $X\times Y$ is not pseudocompact. By Propositions~\ref{erezn2} and~\ref{erezn3} 
the space $X$ is $p$-compact and ${Y}$ is sequentially $q$-compact, 
      and by Proposition~\ref{cpscp} the product space $Y^\om$ is countably compact.
      \end{proof}

      \section{Maximally Homogeneous Spaces}
      
      Given a space $X$, we denote  the group 
      of self-homeomorphisms of $X$ by $\aut X$ and set 
      \[
      H(X)=\set{g(x): x\in X, g\in \aut{\be X}}.
      \]
      
      The following assertion is a direct consequence of these definitions. 
      
      \begin{assertion}\label{mh1}
      If $X$ is a space, then $H(H(X))=H(X)$. 
      If $X$ is a  homogeneous space, then so is $H(X)$.
      \end{assertion}
      
      \begin{definition}\label{defmaxhom}
      We say that a space $X$ is {\it maximally homogeneous} if $X$ 
is  homogeneous and $H(X)=X$.
      \end{definition}
      
      This definition immediately implies the following assertion. 
      
      \begin{assertion}\label{mh2}
      If $X$ is a homogeneous space, then $H(X)$ is maximally homogeneous. 
      \end{assertion}

      \begin{assertion}\label{mh3p1}
      If $X$ is an extremally disconnected space,
      $U,V\subset X$ are non\-empty disjoint open subsets of $X$,
      and $f\colon U\to V$ is a homeomorphism, then there exists an $\tilde f\in \aut{\be X}$ for 
which
      $\tilde f\restriction_{U}=f$, $\tilde f\restriction_{V}=f^{-1}$, 
      and $\tilde f\restriction_{S}=\id_S$, where $S=\be X\setminus (\cl U\cap \cl 
V)$.
      \end{assertion}
      \begin{proof}
       We set $W=U\cup V$. Consider the homeomorphism $g\colon W\to W$ defined by 
      \[
      g(x) = \begin{cases}
      f(x)&\text{if $x\in U$,}
      \\
      f^{-1}(x)&\text{if $x\in V$.}
      \end{cases}
      \]
      Let $\tilde g\colon \be W\to \be W$ be the continuous extension of $g$, and let 
      $\tilde f\colon  \be X\to \be X$ be the homeomorphism defined by  
      \[
      \tilde f(x) = \begin{cases}
       \tilde g(x) &\text{if $x\in \be W$,}
      \\
      x &\text{if $x\in \be X \setminus \be W$.}
      \end{cases}
      \]
      \end{proof}
      
      Statement~\ref{mh3p1} has the following corollary. 
      
      \begin{cor}\label{mh3p2}
      If $X$ is a maximally homogeneous extremally disconnected space,
      $U,V\subset X$ are nonempty disjoint open subsets of $X$,
      and $f\colon U\to V$ is a homeomorphism, then there exists an $\tilde f\in \aut X$ for 
which
      $\tilde f\restriction_{U}=f$, $\tilde f\restriction_{V}=f^{-1}$,
      and $\tilde f\restriction_{S}=\id_S$, where $S= X\setminus (\cl U\cap \cl 
V)$.
      \end{cor}

      \begin{assertion}\label{mh3}
      Let $X$ be a maximally homogeneous nondiscrete extremally disconnected 
space, and let $(x_n)_{n\in\om},(y_n)_{n\in\om}\in \gs_d(X)$.
      Then there exists an $f\in\aut X$ such that $f(x_n)=y_n$ for all $n\in\om$.
      \end{assertion}
      
\begin{proof}
      There exists a nonempty clopen set $O\subset X$ such that
      \[
      O\cap \cl{(x_n)_{n\in\om}}=O\cap \cl{(y_n)_{n\in\om}}=\es.
      \]
      Take $(z_n)_{n\in\om}\in \gs_d(X)$ for which $(z_n)_{n\in\om}\subset 
O$.
      
      Let us show that there exists a $g\in\aut X$ such that $g(x_n)=z_n$ for all 
$n\in\om$.
      For each $n\in\om$, take $g_n\in\aut X$ taking $x_n$ to $z_n$.
      There exist sequences $(U_n)_{n\in\om},(V_n)_{n\in\om}\in\go_d(X)$ satisfying the 
conditions $g_n(U_n)=V_n$, $x_n\in U_n\subset X\setminus W$, and $z_n\in V_n\subset 
W$ for each $n\in\om$. We set $U=\bigcup_{n\in\om}U_n$ and 
$V=\bigcup_{V_n\in\om}U_n$ and consider the homeomorphism $g'\colon U\to V$ defined by
      $g'(x)=g_n(x)$ if $x\in U_n$, $n\in\om$.
      Corollary~\ref{mh3p2} implies the existence of a $g\in \aut X$ for which
      $g\restriction_{U}=g'$, $g\restriction_{V}={g'}^{-1}$, 
      and $g \restriction_{S}=\id_S$, where $S= X\setminus (\cl U\cap \cl V)$.
      
      Similarly, there exists an $h\in\aut X$ such that 
$h(y_n)=z_n$ for all $n\in\om$. 
      It remains to set $f=h^{-1}\circ g$. 
      \end{proof}

      \begin{cor}\label{mh4}
      Let $X$ be a maximally homogeneous extremally disconnected space. 
Then $u(\zeta)=u(\rho)$ for all $\zeta,\rho\in \gs_d(X)$.
      \end{cor}

      \begin{proposition}\label{mh5}
      Let $X$ be a maximally homogeneous extremally disconnected space
      containing a countable nonclosed discrete subspace.
      Then $X$ is discretely $p$-compact for some free 
ultrafilter $p\in \om^*$.
      \end{proposition}
      
\begin{proof}
      Let $\zeta\in\gs_d(X)$ be a nonclosed discrete sequence.
      Then $u(\zeta)\neq\es$. Take $p\in u(\zeta)$.
      Let us show that $X$ is discretely $p$-compact.
      If $\rho\in\gs_d(X)$, then $u(\rho)=u(\zeta)\ni p$ by 
      Corollary~\ref{mh4}. Therefore, 
$\rho$ has a $p$-limit in  $X$.
      \end{proof}
      
      Propositions~\ref{mh5} and~\ref{cpscp} imply  the following  assertion. 
      
      \begin{cor}\label{mh6}
      Let $X$ be a countably compact maximally homogeneous extremally 
disconnected space. Then $X^\omega$ is countably compact.
      \end{cor}

      \begin{lemma}\label{mh7l1}
      Let $X$ be an extremally disconnected space
      which is not a $P$-space. Then 
      there exists a  point $x_*\in X$ and a sequence $(O_n)_{n\in\om}\go(X)$ such that $O_n$ 
is clopen and  $x_*\in O_{n+1}\subset O_n$ for each $n\in\om$
      and the set $\bigcap_{n\in\om}O_n$ is nowhere dense in~$X$.
      \end{lemma}
      
\begin{proof}
      There exists a  point $x_*\in X$ and a sequence $(U_n)_{n\in\om}\go(X)$ such that $x_*\in 
U_{n+1}\subset subset O_n$ and $U_n$ is clopen for each 
      $n\in\om$ and $x_*\notin \int F$, where $F=\bigcap_{n\in\om}O_n$.
      The set $F$ is closed, and since $X$ is extremally disconnected, it follows that  
      $\int F$ is clopen. It remains to set $O_n=U_n\setminus \int 
F$.
      \end{proof}

      \begin{proposition}\label{mh7}
      Let $X$ be a homogeneous extremally disconnected space
      which is not a $P$-space.
      Then $X^\om$ is pseudocompact.
      \end{proposition}
      
\begin{proof}
      Using Lemma~\ref{mh7l1}, we take 
      $x_*\in X$ and $(O_n)_{n\in\om}\go(X)$ such that $O_n$ is clopen 
and $x_*\in O_{n+1}\subset O_n$ for each $n\in\om$
      and the set $\bigcap_{n\in\om}O_n$ is nowhere dense in~$X$.
      
      Let $\zeta_n=(U_{n,k})_{k\in\om}\in \go_n(X)$. By virtue of Proposition~\ref{p1p}, 
      to prove the proposition, it suffices to show that 
$\bigcap_{n\in\om}u(\zeta_n)\ne\es$.
      
      Since $X$ is a zero-dimensional homogeneous space, it follows that, 
      for each $k\in\om$, there exist  homeomorphic clopen  
sets $V_k, V_{0,k},V_{1,k}, ... , V_{k,k}$ such that $x_*\in V_k \subset O_k$, 
      $S_k=V_k\setminus V_{k+1}\neq\es$, and $V_{n,k}\subset U_{n,k}$ for 
$n=0,1,\dots,k$.
      Let $h_{n,k}\colon V_k\to V_{n,k}$ be homeomorphisms, and let $S_{n,k}=h_{n,k}(S_k)$ 
for $n=0,1,\dots,k$. We set
      \[
      W_{n,k} = \begin{cases}
      U_{n,k},  & k < n,
      \\
       S_{n,k}, & k \geq n
      \end{cases}
      \]
      for $n,k\in\om$.
      
      \begin{lemma}
\label{lemma2}
The following assertions hold:
      \begin{itemize}
      \item[\rm(1)] $u((S_k)_{k\in\om})\neq\es$;
      \item[\rm(2)] $u((S_k)_{k\in\om})=u((W_{n,k})_{k\in\om})$ 
      			for any $n\in\om$.
      \end{itemize} 
      \end{lemma}
      
\begin{proof}
      We set $F=\bigcap_{k\in\om}V_k$ and
      $Q_k=\bigcup_{j\geq k}S_j$. By construction  we have $Q_k=V_k\setminus F$.
      
      (1) 
      Since $P=\bigcap_{n\in\om}O_n$ is nowhere dense in $X$
      and $F\subset P$, it follows that $F$ is nowhere dense in $X$.
      Therefore, $\cl{Q_0}=\cl{V_0\setminus F}=V_0$.
      We have  $x_*\in \cl{Q_0}\setminus Q_0$, because $x_*\in F$. 
      Hence $x_*$ is an accumulation point for the sequence 
$(S_k)_{k\in\om}$
      and $u((S_k)_{k\in\om})\neq\es$.
      
      (2) 
      We set $P_n = \bigcup_{k\geq n}S_{n,k}$.
      Consider the homeomorphism $f_n\colon Q_n \to P_n$ defined by $f_n(x)=h_{n,k}(x)$
      if $x\in S_k$.
      Corollary~\ref{mh3p2} implies the existence of a $\tilde 
f_n\in\aut{X}$,
      for which that $\tilde f_n\restriction_{Q_n}=f_n$.
      Since $\tilde f_n(S_k)=S_{n,k}=W_{n,k}$ for $k\geq n$, we have 
      $u((S_k)_{k\in\om})=u((W_{n,k})_{k\in\om})$.
      \end{proof}
      
      The inclusions $W_{n,k}\subset U_{n,k}$ for $n,k\in\om$ and Lemma~\ref{lemma2} 
      imply
      \[
      \bigcap_{n\in\om}u(\zeta_n)\supset
      \bigcap_{n\in\om}u((W_{n,k})_{k\in\om})=
      u((S_k)_{k\in\om})
      \ne\es.
      \]
      \end{proof}
      
      \section{Homogeneous Subspaces of Product Spaces}
      
      \begin{theorem}\label{t1}
      Let $X$ be a homogeneous $\be\om$ space.
      Then any compact subspace of $X$ is finite.
      \end{theorem}
      
\begin{proof}
      Suppose that, on the contrary, $X$ contains an infinite 
compact subspace. 
      By Proposition~\ref{ed4} we have $\sp_k(x,X)=\om^*$ for each 
$x\in X$,
      which contradicts Propositions~\ref{ed6} and~\ref{ed9}.
      \end{proof}
      
      \begin{cor}\label{tc1}
      Any compact subspace of a homogeneous extremally disconnected space is finite.
      \end{cor}
      
      Corollary~\ref{tc1} is also implied by Theorem~2(c) in \cite{Douwen-1979}.
      
      \begin{theorem}\label{t2} (CH)
      Let $Y$ be a $\be\om$ space, and let $X\subset Y^\om$ be a homogeneous 
space.
      Then any compact subspace of $X$ metrizable.
      \end{theorem}
      
\begin{proof}
      Suppose that, on the contrary, $X$ contains a nonmetrizable compact subspace $K$. 
Then Statement~\ref{ed11} 
implies the existence of a sequence $\zeta=(z_m)_{m\in\om}\in \gs_{dk}(X)$ such that 
$\cl\zeta$ is homeomorphic to $\be\om$. Take $z\in X$. For each $p\in \om^*$, 
      we fix $f_p\in\aut X$ such that $f_p(\lim_p\zeta)=z$.
      
      By virtue of Proposition~\ref{ed10} there exists a $C\subset \om^*$, 
$|C|=2^{2^\om}$, consisting of pairwise
      $\lkr$-incomparable selective ultrafilters.
      
      Statement~\ref{ed12} implies that, for each $p\in C$, there exists 
      an $m_p\in \om$ and an $N_p\in p$ such that 
      the set $\set{\pi_{m_p}(z_n):n\in N_p}$ is discrete and 
$\pi_{m_p}(x_j)\ne \pi_{m_p}(x_i)$
      for any different $i,j\in N_p$.
      Since $|C|>\om$, there are different $p,q\in C$ for which 
$m=m_p=m_q$.
      Therefore, $p,q\in \sp_k(\pi_m(z),X)$. This contradicts Proposition~\ref{ed6}.
      \end{proof}
      
      \begin{cor}\label{tc2}  
\textup{(CH)}
      Let $Y$ be a  $\be\om$ space, and let $X\subset Y^\om$ be a homogeneous compact 
space. Then $X$ is metrizable.
      \end{cor}
      
      \begin{theorem}\label{t3} 
\textup{(CH)}
      Let $Y$ be a $\be\om$ space, and let $X$  be a homogeneous 
subspace of  $Y^n$ for some $n\in\om$.
      Then any compact subspace of $X$ is finite.
      \end{theorem}
      
\begin{proof}
      Let $K\subset X$ be a compact set. Then  $K$ 
is metrizable by Theorem~\ref{t2}.
      Since all metrizable compact sets in any $\be\om$ space are finite, it follows that
      the projection of $K$ on any factor in $Y^n$ is finite. Therefore, $K$ 
is finite.
      \end{proof}

      \begin{question} \label{q1}
      Are Theorem~\ref{t2} and~\ref{t3}
      true na\"{\i}vely? What if $Y$ is additionally assumed to be an F-space 
(an \ed/ space)?
      \end{question}
      
      In what follows, we treat positive integers as sets of smaller integers. 
      
      Let $X\subset Y^n$, where $X$ and $Y$ are sets and $k,n\in\om$.
      We write $\dim_p(X)\geq k$ if there exists a set $M\subset n$, $|M|=k$,
      and a sequence $(x_i)_{i\in\om}\subset X$ such that $\pi_m(x_i)\neq \pi_m(x_j)$ for 
any $m\in M$ 
      and any different $i,j\in\om$. We set
      \[
      \dim_p(X)=\max \set{m: \dim_p(X)\geq m}.
      \]
      Let $\cp(Y,n,k)$ denote the family of sets of the form $\sset p\times 
Y^M$,
      where $M\subset n$, $|M|=k$, and $p\in Y^{n\setminus M}$.
      We set 
      \[
      \cp^*(Y,n,k)=\set{\bigcup\gamma: \gamma\subset \cp(Y,n,k), 
|\gamma|<\om}.
      \]
      Given $q\in Y^n$, we put 
      \begin{align*}
      \cp_q(Y,n,k)&=\set{
      	\sset{q\restriction_{n\setminus M}}\times Y^M: 
      	M\subset n, |M|=k
      },
      \\
      \cp_q^*(Y,n,k)&=\bigcup \cp_q(Y,n,k).
      \end{align*}
      
\begin{assertion} 
\label{at3}
      Suppose that $Y$ is an infinite set, $n\in\om$, $k\leq n$, and $X\subset Y^n$.
Then the following assertions hold:
      \begin{itemize}
      \item[\rm(1)] if $Z\subset X$, then $\dim_p(Z)\leq \dim_p(X)$; 
      \item[\rm(2)] if $X_1,X_2\subset Y^n$ and $X=X_1\cap X_2$, then 
      $$
\dim_p(X)=\max(\dim_p(X_1),\dim_p(X_2));
$$
      \item[\rm(3)] $\dim_p(Z)= k$ for all $Z\in \cp^*(Y,n,k)$;
      \item[\rm(4)] given $q,r\in Y^n$, $q\in \cp_r(Y,n,k)$ if and only if 
      $|\set{i<n: \pi_i(q)\neq \pi_i(r)}|\leq k$;
      \item[\rm(5)] 
      $\dim_p(X)\leq k$ if and only if
      $X\subset M$ for some $M\in \cp^*(Y,n,k)$;
      \item[\rm(6)] 
      if $Y$ is a space, $q\in X$,
      	and $\dim_p(X)\leq k$, then
      $q\in \intx X{X\cap \cp_q^*(Y,n,k)}$.
      \end{itemize} 
      \end{assertion}
      
\begin{proof}
      Assertion (1) is obvious. 

Assertion (2) follows from the observation that if $\dim_p(X) \geq k$, then, by definition, 
      we have either $\dim_p(X_1) \geq k$ or $\dim_p(X_2) \geq k$.
      
Let us prove (3). Obviously, we have $\dim_p(Q)=k$ for $Q\in \cp(Y,n,k)$, so that it suffices to  
apply (2).
      
Assertion (4) follows from the definitions.
      
      Let us prove (5). Assertions (1) and (3) imply that if 
      $X\subset M$ for some $M\in \cp^*(Y,n,k)$,
      then $\dim_p(X)\leq k$. Let us prove that if $\dim_p(X)\leq k$, then 
      $X\subset M$ for some $M\in \cp^*(Y,n,k)$. Assume the contrary. 
      Let us construct  a sequence $(q_n)_{n\in \om}\subset X$ by induction on $n$ as follows. 
      Take any $q_0\in X$. Having constructed $q_0,...,q_{n-1}\in X$, we choose 
      $q_n\in X\setminus \bigcup_{i=0}^{n-1} \cp_{q_i}^*(Y,n,k)$.
      Note that  $q_m\notin \cp_{q_n}^*(Y,n,k)$ for $m,n\in \om$.
Consider the set $M_{n,m}=\set{i<n:\pi_i(q_n)\neq \pi_m}$; 
      according to assertion~(4), we have $|M_{n,m}|>k$. 
      Let $\cm$ denote the family of all finite sets $n$. We color the complete graph 
on the $\om$
      elements of $\cm$ by assigning the set $M_{n,m}$ to each edge 
$\sset{n,m}$. According to Ramsey's theorem, there exists an infinite set $M\subset 
\om$ such that
      $M=M_{n,m}=M_{n',m'}$ for $n,m,n'm'\in \om$, $n\neq m$, $n'\neq m'$.
      By definition, we have $\dim_p(X)\geq |M|>k$.
      
      Let us prove (6). 
      Assertion (5) implies the existence of a finite set $\gamma\subset \cp(Y,n,k)$
      for which $X\subset \bigcap \gamma$. Since $\gamma$ consists of 
closed sets, it follows that $x\in \intx X{X\cap \bigcap \gamma_1}$ and $x\notin 
\cl{X\cap \bigcap \gamma_1}$, where $\gamma_1=\set{F\in\gamma: q\in F}$ and 
$\gamma_1=\set{F\in\gamma: q\notin F}$. The inclusion $\gamma_1\subset 
\cp_q(Y,n,k)$ implies 
      $q\in \intx X{X\cap \cp_q^*(Y,n,k)}$.
      \end{proof}
      
      \begin{proposition}\label{tp4-1}
      Suppose that $Y$ is an F-space, $n\in\om$, $n>1$, and $X\subset Y^n$ is a homogeneous 
space.
      Then $\dim_p(K)<n-1$ for any compact subspace $K\subset 
X$.
      \end{proposition}
      \begin{proof}
     Assume the contrary. Then $\dim_p(K)\geq n-1$. There exists a set $M\subset 
n$ with $|M|=n-1$ and a sequence 
      $\zeta=(x_i)_{i\in\om}\subset K$ such that $\pi_m(x_i)\neq \pi_m(x_j)$ 
for $m\in M$.
      We can assume without loss of generality that $M=\sset{1,2,\dots ,n-1}$.
      Passing to a subsequences if necessary, we can also assume that
      the sequences $\zeta_i=(\pi_i(x_j))_{j\in\om}$ are discrete and $\zeta_i\in\gs_{dk}(X)$ 
for $i>0$.
      Proposition~\ref{ed9} implies the existence of three pairwise $\lkr$-incomparable weak $P$ 
ultrafilters $p$, $q$, and $r$. 
      Let $x=\lim_p \zeta=(x_0,\dots,x_{n-1})$.
      By  virtue of Proposition~\ref{ed4} we have $\sp_k(x,X)=\om^*$.
      There exist $\xi=(u_n)_{n\in\om}$ and $\rho=(v_n)_{n\in\om}\in \gs_{dk}(X)$
      for which $x=\lim_q \xi=\lim_r \rho$.
      According to Proposition~\ref{ed8}, for $i=1,2,\dots,n-1$, 
      we have $M_i=\set{n\in\om: \pi_i(u_n)=x_i}\in q$ and 
      $N_i=\set{n\in\om: \pi_i(v_n)=x_i}\in r$. 
      We set $M=\bigcap_{i=1}^{n-1} M_i$ and $N=\bigcap_{i=1}^{n-1} N_i$.
      Then $M\in q$, $N\in r$,  $\set{u_n: \in M}\subset F$,  and 
      $\set{v_n: \in N}\subset F$, where $F=\set{(w,x_1,\dots , x_{n-1}: w\in Y}$.
      The space $F$ is homeomorphic to $Y$, and $q,r\in \sp_k(x,F)$.
      This contradicts Proposition~\ref{ed6}.
      \end{proof}
      
      \begin{proposition}\label{tp4-2}
      If $Y$ is a $\be\om$ space, $n\in\om$, and $X\subset Y^n$ is homogeneous, 
then $\dim_p(K)>1$ for any infinite compact space 
$K\subset X$.
      \end{proposition}
      
\begin{proof}
      Suppose that, on the contrary, $X$ contains an infinite compact set and 
      $\dim_p(K)\leq 1$ for any compact space $K\subset X$.
      Let $q\in X$. By Proposition~\ref{ed4} we have 
$\sp_k(q,X)=\om^*$.
      Proposition~\ref{ed9} implies the existence of $n+1$ pairwise 
incomparable ultrafilters $r_0,r_1,\dots,r_n\in \om^*$.
      For each $i<n+1$, choose $\zeta_i=(x_{i,k})_{k\in\in\om}\in\gs_{dk}(X)$ so that
      $q=\lim_{r_i}\zeta_i$. For $K=\bigcup_{i=0}^n \cl{\zeta_i}$, we have 
$\dim_p K=1$.
      Let $\cp_q(X,n,1)=\sset{F_0,F_1,...,F_{n-1}}$.
      It follows from Statement~\ref{at3}(6) that if $i\leq n$, then 
      $\set{k\in\om: x_{i,k}\in F_{m_i}}\in\zeta_i$ for some $m_i<n$.
      Thus, $m=m_i=m_j$ for some $i<j<n$, so that 
$r_i,r_j\in\sp_k(q,F_m)$. Since $F_m$ is homeomorphic to $Y$ and $Y$ is a $\be\om$ 
space, we have obtained a contradiction to Proposition~\ref{ed6}.
      \end{proof}

      \begin{theorem}\label{t4} 
      Let $Y$ F-space, and let $X\subset Y^3$ be a homogeneous space.
      Then any compact subspace of $X$ is finite.
      \end{theorem}
      
\begin{proof}
      Suppose that, on the contrary, $X$ contains an infinite 
compact set $K$.
      Then Proposition~\ref{tp4-1} implies $\dim_p K < 3-1=2$, and 
      Proposition~\ref{tp4-2} implies $\dim_p K > 1$.
      \end{proof}
      
      \begin{question} \label{q2}
      Let $Y$ be an F-space, and let $X\subset Y^4$ be a homogeneous space.
      Is it true that any compact subspace of $X$ is finite?
      \end{question}
      
      \begin{proposition}\label{tp5}
      Suppose that $Y$ is a space, $n\in\om$,  $X\subset Y^n$ is a zero-dimensional 
homogeneous compact space, and $k=\dim_p X$. 
      Then there exists a finite space $Z$ such that $X$ is embedded in 
$Z\times Y^k$ and $(Z\times Y)^k$.
      \end{proposition}
 
     \begin{proof}
According to Statement~\ref{at3}(5), we have  $X\subset \bigcup\gamma$ for 
some finite set $\gamma\subset \cp(Y,n,k)$. The family $\cp(Y,n,k)$ consists of 
closed subsets; therefore, $\intx X{X\cap F}\neq \es$ for some 
$F\in\gamma$.
      
Since $X$ is a homogeneous zero-dimensional compact set, it follows that there exists a finite 
partition $\lambda$ of $X$ into clopen subsets such that 
each $U\in\lambda$ is embedded in $\intx X{X\cap F}\subset F$. The space $F$ 
is homeomorphic to $Y^k$; hence $X$ is embedded in $Z\times Y^k$, where $Z$ is a finite discrete 
space of cardinality $|Z|=|\gamma|$.
      \end{proof}

      \begin{theorem}\label{t5} 
      If $Y$ is an F-space, $n\in\om$, and $X\subset Y^n$ is a homogeneous 
compact space, then $X$ is finite.
      \end{theorem}
      
\begin{proof}
      Suppose that, on the contrary, $X$ is infinite. We can assume that $n$ is the least 
positive integer $m$ such that $X$ is embedded in the $m$th power of an F-space. 
All F-spaces are zero-dimensional; therefore, 
      $\dim_p X=n$ by Proposition~\ref{tp5}. This contradicts Proposition~\ref{tp4-1}.
      \end{proof}
      
      \bibliographystyle{amsplain}

\begin{thebibliography}{8}
      
      
      \bibitem {cr1966}
          W.~W.~Comfort and Kenneth~A.~Ross,
          Pseudocompactness and uniform continuity in topological groups, 
          Pacific J. Math. 16,  483--496  (1966).
      
      
      \bibitem {cm1985}
      W.~Comfort and Jan~van Mill, On the product of homogeneous 
spaces. Topol. Appl. 21, 297--308 (1985).
      
      \bibitem {cm1987}
      W.~W.~Comfort and Jan~van~Mill, A homogeneous extremally disconnected 
countably compact space, Topol. Appl. 25 (1),  65--73 (1987).
      
      \bibitem {k1994}
      Akio Kato, A new construction of extremally disconnected topologies, 
Topol. Appl. 58 (1), 1--16 (1994).
      
      \bibitem {ls1997}
      W.~Lindgren and A.~Szymanski,  A Non-Pseudocompact 
Product of Countably Compact Spaces Via Seq, Proc. Amer. 
Math. Soc. 125 (12), 3741--3746 (1997).  
      
      
      \bibitem {fro1967}
      Zden\v{e}k Frol\'{\i}k, Homogeneity problems for extremally disconnected 
spaces, Comment. Math. Univ. Carolinae 008.4, 
757--763 (1967). 
      
      \bibitem {at2009}
      Alexander Arhangel'skii and Mikhail Tkachenko, Topological 
Groups and Related Structures (Atlantis Press/World Sci., Paris, 2008).
      
      \bibitem{arh67}
      A.~V.~Arhangel'skii, Groupes topologiques extremalement discontinus, C.~R.~Acad. Sci. 
Paris 265, 822--825 (1967).
      
      
      \bibitem {rezn} 
E.~A.~Reznichenko, Homogeneous products of spaces, 
Mosc. Univ. Math. Bull. 51 (3), 6--8 (1996). 
      
      \bibitem {AT} 
A.~Arhangel'ski\u{\i}, \textit{ Every extremally 
disconnected bicompactum is inhomogeneous}, Soviet Math. Dokl. 8, 897--900 (1967).
      
      
      \bibitem {Gli} 
I.  Glicksberg, 
Stone--\v{C}ech Compactifications of 
Products, in: The Mathematical Legacy of Eduard 
\v{C}ech, ed. by M.~Kat\v{e}tov and P.~Simon (Birkh\"auser,  Basel,   1993).
      
      \bibitem {gli-jer-1960} 
L. Gillman and M. Jerison, Rings of Continuous
      Functions (Van Nostrand Reinhold, New York, 1960).
      
      \bibitem {bal-dow-1991}
      B. Balcar and A. Dow, Dynamical
      systems on compact extremally disconnected
      spaces, Topology Appl. 41, 41--56 (1991).
      
      \bibitem {com-neg-1991}
      W.~W.~Comfort and S.~Negrepontis,
      The Theory of Ultrafilters
      (Springer, Berlin, 1974).
      
      \bibitem {van-douwen-1981}
      E.~K.~van~Douwen, Prime mappings, number of factors and binary 
operations, Dissertationes Mathematicae 199 (1981).
      
      \bibitem {kunen-1990}
      K.~Kunen, Large homogeneous compact spaces, in: Open Problems in Topology 
(North-Holland, Amsterdam, 1990), pp.~261--270.
      
      \bibitem {Shelah-Rudin-1979}
      Saharon~Shelah and M.~E.~Rudin, Unordered types of ultrafilters, 
Topol. Proc. 3, 199--204 (1979). 
      
      \bibitem {kunen-1978}
      K.~Kunen, Weak $P$-points in $N^*$, Coll. Math. Soc. J\'anos
      Bolyai 23, 741--749 (1978).
      
      \bibitem {Simon-1985}
      P.~Simon, Applications of independent linked families, 
      Colloq. Math. Soc. J\'anos Bolyai, 41, 561--580 (1985).
      
      \bibitem {Rudin-1956}
      W.~Rudin, 
      Homogeneity problems in the theory of \v{C}ech compactifications, 
      Duke Math. J. 23, 409--420 (1956).
      
      
      \bibitem {Rudin-1971}
       M.~E.~Rudin, Partial orders on the types in $\beta N$, 
       Trans. Amer. Math. Soc. 155 (2), 353--362 (1971).
       
      
      \bibitem {Douwen-1979}
       Eric~K.~van~Douwen,
      Homogeneity of $\be G$ if $G$ is a topological group, 
      Colloq. Math. 41, 193--199 (1979).
      
      \bibitem{Ramsey}
      F.~P. Ramsey, On a problem of formal logic, Proc. London Math. Soc. 30, 
264--286 (1930).
      
      
      \end{thebibliography}

      \end{document}